\documentclass[12pt]{article}

\usepackage{amssymb}
\usepackage{amsmath}
\usepackage{amsthm}
\usepackage{enumerate}
\usepackage[usenames]{color}
\usepackage{graphicx}

\newtheoremstyle{plain}%
  {8pt plus2pt minus4pt}%
  {8pt plus2pt minus4pt}%
  {\itshape}%
  {}%
  {\bfseries\scshape}%
  {}%
  {6pt}% Space after theorem head
  {}%

\newtheoremstyle{remark}%
  {8pt plus2pt minus4pt}%
  {8pt plus2pt minus4pt}%
  {\upshape}% Body font
  {}%
  {\bfseries\scshape}%
  {}%
  {6pt}% Space after theorem head
  {}%

\theoremstyle{plain}
\newtheorem{thm}{Theorem}[section]

\newtheorem{clm}[thm]{Claim}

\theoremstyle{remark}
\newtheorem{rem}[thm]{Remark}

\newcommand{\e}{\varepsilon}

\newcommand{\eps}{\varepsilon}

\def\l{\lambda}

\def\r{\rho}
\def\a{\alpha}
\def\om{\omega}
\newcommand{\brac}[1]{\left(#1\right)}
\newcommand{\bfrac}[2]{\brac{\frac{#1}{#2}}}
\newcommand{\set}[1]{\left\{#1\right\}}

\title{Flips in Graphs}

\author{Tom Bohman\thanks{
\footnotesize {Department of Mathematical Sciences, Carnegie Mellon University, Pittsburgh, PA 15213, 
\texttt{\{tbohman,adudek,af1p,pikhurko\}@andrew.cmu.edu}}}~\thanks{\footnotesize{Research partially 
supported by NSF grant DMS-0401147}}
\and
Andrzej Dudek\footnotemark[1]
\and
Alan Frieze\footnotemark[1]~\thanks{\footnotesize{Research partially supported by NSF grant DMS-0753472}}
\and
Oleg Pikhurko\footnotemark[1]~\thanks{\footnotesize{Research partially supported by NSF grant DMS-0457512}}
}

\begin{document}
\maketitle

\begin{abstract}
We study a problem motivated by a question related to quantum-error-correcting codes.
Combinatorially, it involves the following graph parameter:
$$f(G)=\min\set{|A|+|\{x\in V\setminus A : d_A(x)\text{ is odd}\}| : A\neq\emptyset},$$
where $V$ is the vertex set of $G$ and $d_A(x)$ is the number of neighbors of~$x$ in~$A$.
We give asymptotically tight estimates of~$f$ for the random graph $G_{n,p}$ when $p$ is constant. Also,
if
$$f(n)=\max\set{f(G):\;|V(G)|=n}$$
then we show that $f(n)\leq (0.382+o(1))n$.
\end{abstract}

\section{Introduction}
In this paper we consider a problem which is motivated by a question from quantum-error-correcting codes. To see how to use graphs to construct quantum-error-correcting codes see, \textit{e.g.},  \cite{HDERNB, LYGG, YCO}.
%\emph{diagonal distance} on page~8 in~\cite{LYGG} when $D=2$).

Given a graph $G$ with $\pm1$ signs on vertices, each vertex can perform at most
one of the following three operations: $O_1$ (flip all of its neighbors, \textit{i.e.}, change
their signs), $O_2$ (flip itself), and $O_3$ (flip itself and all of its
neighbors). We want to start with all $+1$'s, execute some non-zero number of operations and return to all $+1$'s. 
The \emph{diagonal distance} $f(G)$ is the minimum
number of operations needed (with each vertex doing at most one operation).

Trivially, 
\begin{equation}\label{trivial}
f(G)\le \delta(G)+1
\end{equation}
holds, where $\delta(G)$ denotes the minimum degree. Indeed, a vertex with the minimum degree applies $O_1$ and then its neighbors fix 
themselves applying $O_2$. 
Let $$f(n) = \max f(G),$$ where the maximum is taken over all non-empty graphs of order~$n$. 
Shiang  Yong Looi (personal communication) asked for a good approximation on $f(n)$.

In this paper we asymptotically determine the diagonal distance of the random graph $G_{n,p}$ for any $p\in(0,1)$. 

We denote the \textit{symmetric difference} of two sets~$A$ and~$B$ by $A\bigtriangleup B$ and the \textit{logarithmic function} with base~e as~$\log$. 
\begin{thm}\label{thm:1}
There are absolute constants $\lambda_0\approx 0.189$ and $p_0\approx 0.894$, see~\eqref{const:lambda} and~\eqref{const:p}, such that for 
$G=G_{n,p}$ asymptotically almost surely:
\begin{enumerate}[(i)]
\item $f(G) = \delta(G)+1$ for $0<p<\lambda_0$ or $p=o(1)$,
\item $|f(G) - \lambda_0 n|=\tilde{O}(n^{1/2})$ for  $\lambda_0\le p\le p_0$,
\item $f(G)=2+\min_{x,y\in V(G)} \left|\left(N(x)\bigtriangleup N(y)\right)\setminus\{x,y\}\right|$ for $p_0<p<1$ or $p=1-o(1)$.
\end{enumerate}
\end{thm}
\noindent
(Here $\tilde{O}(n^{1/2})$ hides a polylog factor).

Figure~\ref{fig:p} visualizes the behavior of the diagonal distance of~$G_{n,p}$.
In addition to Theorem~\ref{thm:1} we find the following upper bound on~$f(n)$.
\begin{thm}\label{thm:2}
$f(n)\le (0.382+o(1))n$.
\end{thm}

\begin{figure}
\centering
\begin{picture}(0,0)%
\includegraphics{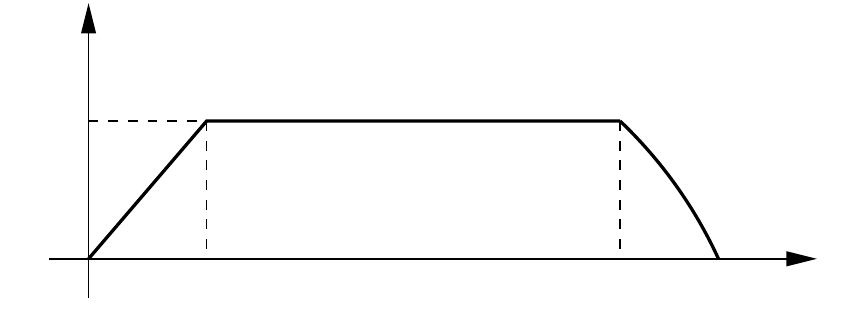}%
\end{picture}%
\setlength{\unitlength}{4144sp}%
\begingroup\makeatletter\ifx\SetFigFont\undefined%
\gdef\SetFigFont#1#2#3#4#5{%
  \reset@font\fontsize{#1}{#2pt}%
  \fontfamily{#3}\fontseries{#4}\fontshape{#5}%
  \selectfont}%
\fi\endgroup%
\begin{picture}(3921,1411)(4816,-5690)
\put(8461,-5641){\makebox(0,0)[lb]{\smash{{\SetFigFont{12}{14.4}{\rmdefault}{\mddefault}{\updefault}{\color[rgb]{0,0,0}$p$}%
}}}}
\put(8056,-5641){\makebox(0,0)[lb]{\smash{{\SetFigFont{12}{14.4}{\rmdefault}{\mddefault}{\updefault}{\color[rgb]{0,0,0}$1$}%
}}}}
\put(7561,-5641){\makebox(0,0)[lb]{\smash{{\SetFigFont{12}{14.4}{\rmdefault}{\mddefault}{\updefault}{\color[rgb]{0,0,0}$p_0$}%
}}}}
\put(5671,-5641){\makebox(0,0)[lb]{\smash{{\SetFigFont{12}{14.4}{\rmdefault}{\mddefault}{\updefault}{\color[rgb]{0,0,0}$\lambda_0$}%
}}}}
\put(5041,-5641){\makebox(0,0)[lb]{\smash{{\SetFigFont{12}{14.4}{\rmdefault}{\mddefault}{\updefault}{\color[rgb]{0,0,0}$0$}%
}}}}
\put(4996,-4876){\makebox(0,0)[lb]{\smash{{\SetFigFont{12}{14.4}{\rmdefault}{\mddefault}{\updefault}{\color[rgb]{0,0,0}$\lambda_0$}%
}}}}
\put(4816,-4426){\makebox(0,0)[lb]{\smash{{\SetFigFont{12}{14.4}{\rmdefault}{\mddefault}{\updefault}{\color[rgb]{0,0,0}$\hat{f}(p)$}%
}}}}
\end{picture}%
\caption{The behavior of $\hat{f}(p)=\lim_{n\to\infty}f(G_{n,p})/n$ as a function of~$p$.}\label{fig:p}
\end{figure}
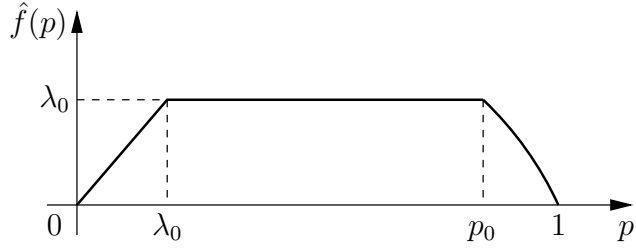

In the remainder of the paper we will use a more convenient restatement of~$f(G)$. 
Observe that the order of execution of operations does not affect the final outcome.
For any $A\subset V=V(G)$,
let $B$ consist of those vertices in $V\setminus A$ that have odd number of
neighbors in $A$. Let $a=|A|$ and $b=|B|$. We want to minimize $a+b$ over all
non-empty $A\subset V(G)$. 
The vertices of $A$ do an $O_1/O_3$ operation, depending on the even/odd parity of their neighborhood in~$A$.
The vertices in $B$ then do an $O_2$-operation to change back to $+1$.

\section{Random Graphs for $p=1/2$}\label{sec:random}
Here we prove a special case of Theorem~\ref{thm:1} when $p=1/2$. This case is somewhat easier to handle.

Let $G=G_{n,1/2}$ be a binomial random graph. First we find a lower bound on $f(G)$. If we choose a non-empty
$A\subset V$ and then generate $G$, then the distribution of $b$ is binomial with parameters $n-a$ and $1/2$, which we denote here by
$Bin(n-a,1/2)$. 
Hence, if  $l$ is such that
\begin{equation}\label{eq:a}
\sum_{a=1}^{l-1} \binom{n}{a}\Pr\left(Bin(n-a,1/2)\le l-1-a\right) =o(1),
\end{equation}
then asymptotically almost surely the diagonal distance of $G$ is at least $l$.

Let $\lambda=l/n$ and $\alpha=a/n$. We can approximate the summand in~\eqref{eq:a}
by
\begin{equation}\label{eq:summand}
2^{n\left(H(\alpha)+(1-\alpha)\left(H\left(\frac{\lambda-\alpha}{1-\alpha}\right)-1\right) +O(\log n/n)\right)},
\end{equation}
where $H$ is the binary entropy function defined as $H(p)=-p\log_2{p}-(1-p)\log_2(1-p)$. For more information about  the entropy function and 
its 
properties  see, \textit{e.g.},~\cite{AS}.
Let 
\begin{equation}\label{defg}
g_\l(\alpha) = H(\alpha)+(1-\alpha)\left(H\left(\frac{\lambda-\alpha}{1-\alpha}\right)-1\right).
\end{equation}
The maximum of $g_\l(\alpha)$ is attained exactly for $\alpha=2\lambda/3$, since
$$g_\l'(\alpha)=\log_2 \frac{2(\lambda-\alpha)}{\alpha}.$$
Now the function 
\begin{equation}\label{defh}
h(\lambda)=g_\l(2\l/3)
\end{equation} 
is concave on $\lambda\in[0,1]$ since 
$$h''(\lambda)=\frac{1}{(\lambda-1)\lambda\log 2}<0.$$ 
Moreover,  
observe that $h(0)=-1$ and $h(1)=H(2/3)-1/3>0$. Thus the equation $h(\l)=0$ has a unique solution $\l_0$ and one can compute that
\begin{equation}\label{const:lambda}
\lambda_0 = 0.1892896249152306\ldots
\end{equation}

Therefore, if $\lambda=\lambda_0-K\log n/n$ for large enough $K>0$, then the left hand side of~\eqref{eq:a} goes to zero and similarly for 
$\lambda=\lambda_0+K\log n/n$ it goes to infinity. In particular,  $f(G)> (\lambda_0-o(1))n$ asymptotically almost surely.

Let us show that this constant $\lambda_0$ is best possible, \textit{i.e.}, asymptotically almost surely $f(G)<(\lambda_0+K\log n/n)n$.
Let $\lambda=\lambda_0+K\log n/n$, $n$ be large, and $l=\lambda n$. Let
$\alpha=2\lambda/3$ and $a=\lfloor \alpha n\rfloor$. We pick a random $a$-set $A\subset V$
and compute $b$. Let $X_A$ be an indicator random variable so that $X_A=1$ if and only if $b=b(A)\le l-a$. Let $X=\sum_{|A|=a} X_A$.
We succeed if $X>0$.

The expectation $E(X)=\binom{n}{a} \Pr\left(Bin(n-a,1/2)\le l-a\right)$ tends to infinity, by our choice of $\l$.
We now show that $X>0$ asymptotically almost surely by using the Chebyshev inequality.
First note that for $A\cap C \neq \emptyset$ we have 
$$Cov(X_A,X_C)=\Pr(X_A=X_C=1)-\Pr(X_A=1)\Pr(X_C=1)=0.$$ 
Indeed, if $x\in V\setminus (A\cup C)$, 
then $\Pr(x\in B(A)|X_C=1)=1/2$, since $A\setminus C\neq\emptyset$ and no adjacency between~$x$ and all vertices in~$A\setminus C$ is exposed 
by 
the event~$X_C=1$.
Similarly, if $x\in C\setminus A$, then $A\cap C\neq\emptyset$ and an adjacency between~$x$ and $A\cap C$ is independent of the 
occurrence of $X_C=1$. This implies that 
$\Pr(x\in B(A)\mid X_C=1)=1/2$ as well. Thus
$\Pr(X_A=1|X_C=1)=\Pr\left(Bin(n-a,1/2)\le l-a\right) = \Pr(X_A=1)$, and consequently, $Cov(X_A,X_C)=0$.

Now consider the case when $A\cap C=\emptyset$. Let $s$ be a vertex in~$A$. Define a new indicator random variable $Y$ which takes the value~$1$ 
if 
and only if $|B(C)\setminus\{s\}|\le l-a$. Observe  that
\begin{align*}
\Pr(Y=1)&=\Pr\left(Bin(n-a-1,1/2)\le l-a\right)\\
           &\le 2\Pr\left(Bin(n-a,1/2)\le l-a\right) = 2\Pr(X_A=1).
\end{align*}
Moreover, 
$$\Pr(X_A=1|Y=1) = \Pr\left(Bin(n-a,1/2)\le l-a\right) = \Pr(X_A=1),$$
since for every $x\in V\setminus A$ the adjacency between~$x$ and~$s$ is not influenced by~$Y=1$.
Finally note that $X_C\le Y$. Thus, 
\begin{align*}
Cov(X_A,X_C)&\le \Pr(X_A=X_C=1)\le \Pr(X_A=Y=1)\\
                      &= \Pr(Y=1)\Pr(X_A=1|Y=1)\le 2\left(\Pr(X_A=1)\right)^2.
\end{align*}
Consequently,
\begin{align*}
Var(X) &= E(X) + \sum_{A\cap C\neq\emptyset, A\neq C} Cov(X_A,X_C) + \sum_{A\cap C = \emptyset} Cov(X_A,X_C)\\
         &\le E(X) +  2\sum_{A\cap C = \emptyset} \left(\Pr(X_A=1)\right)^2\\
         &= E(X) + 2\binom{n}{a}\binom{n-a}{a} \left(\Pr(X_A=1)\right)^2 = o\left(E(X)^2\right),
\end{align*}
as $E(X) = \binom{n}{a} \Pr(X_A=1)$ tends to infinity and $\binom{n-a}{a}=o\left(\binom{n}{a}\right)$. Hence, Chebyshev's inequality yields that 
$X>0$ asymptotically almost surely.

\begin{rem} 
A version of the well-known Gilbert-Varshamov bound (see, \textit{e.g.},~\cite{Lint}) states that if 
\begin{equation}\label{eq:gv}
2^{-n}\sum_{i=1}^{l-1} \binom{n}{i}3^i <1,
\end{equation}
then $f(n)\ge l$. Observe that this is consistent with bound~\eqref{eq:a}. 
Let $\lambda=l/n$. We can approximate the left hand side of~\eqref{eq:gv} by
$$
2^{n\left( H(\lambda) + \lambda\log_2{3} -1 +o(1)  \right)}.
$$
One can check after some computation that 
$$
H(\lambda) + \lambda\log_2{3} -1 = g_{\lambda}(2\lambda/3).
$$
Therefore, \eqref{eq:a} and~\eqref{eq:gv} give asymptotically the same lower bound on~$f(n)$.
\end{rem}

\section{Random Graphs for Arbitrary $p$}
Let $G=G_{n,p}$ be a random graph with constant $p\in (0,1)$.

Observe that for a fixed set~$A\subset V$, $|A|=a$, the probability that a vertex from $V\setminus A$ belongs to~$B(A)$ is 
$$
p(a) = \sum_{0\le i<\frac{a}{2}} \binom{a}{2i+1} p^{2i+1} (1-p)^{a-(2i+1)}
   = \frac{1-(1-2p)^a}{2}.
$$
(If this is unfamiliar, expand $(1-2p)^n$ as $((1-p)-p)^n$ and compare).

\subsection{$0<p< \lambda_0$}\label{sec:case1}
For $p<\lambda_0$ we begin with the upper bound $f(G)\le \delta(G)+1$, see \eqref{trivial}. For the lower bound it is enough to show that 
\begin{equation}\label{small}
\sum_{2\le a\le pn} \binom{n}{a}\Pr\left(Bin(n-a,p(a))\le pn-a\right) =o(1),
\end{equation}
since $\delta(G)+1 \leq np$ 
asymptotically almost surely. (We may assume that $p=\Omega\left(\frac{\log n}{n}\right)$; for otherwise $\delta(G)=0$ with high probability and the theorem is trivially true.) This implies that if $|A|+|B| \le pn$, then $|A|=1$. 

\subsubsection{$p$ Constant}
We split this sum into two sums for $2\le a\le \sqrt{n}$ and $\sqrt{n}<a\le pn$, respectively.
Let $X=Bin(n-a,p(a))$ and  
\begin{equation}\label{defeps}
\eps =1-\frac{pn-a}{(n-a)p(a)} \ge 1-\frac{p}{p(2)} = 1-\frac{1}{2-2p}>0.
\end{equation}
Thus, by Chernoff's bound, 
\begin{equation}\label{Chernoff}
\Pr(Bin(N,\r)\leq (1-\theta)N\r)\leq e^{-\theta^2N\r/2}
\end{equation}
we see that
\begin{align*}
\Pr\left(Bin(n-a,p(a))\le pn-a\right)     &= \Pr\left( X\le (1-\eps)E(X) \right)\\
                                       &\le \exp\{-\eps^2E(X)/2\}\\
                                       &= \exp\{ -\Theta(n)\},
\end{align*}
and consequently,
\begin{align*}
\sum_{2\le a < \sqrt{n}} \binom{n}{a}\Pr\left(Bin(n-a,p(a))\le pn-a\right) &\le \sqrt{n} \binom{n}{\sqrt{n}} \exp\{ -\Theta(n)\}\\ 
                                                     &\le \exp \{ O(\sqrt{n}\log n)\} \exp\{-\Theta(n)\}\\
                                                                                              & =o(1).
\end{align*}
Now we bound the second sum corresponding to $\sqrt{n}<a\le pn$. Note that  
\begin{align*}\sum_{\sqrt{n}\le a \le pn} \binom{n}{a}&\Pr\left(Bin(n-a,p(a))\le pn-a\right)\\
&=\sum_{\sqrt{n}\le a \le pn} \binom{n}{a}\Pr\left(Bin\left(n-a,\frac12+O(e^{-\Omega(n^{1/2})})\right)\le pn-a\right)\\
&\leq n2^{nh(p)+o(1)}=o(1).
\end{align*}
Here $h$ is defined in \eqref{defh} and the right hand limit is zero since $p < \lambda_0$.

\subsubsection{$p=o(1)$}\label{sec:p_little_1}
We follow basically the same strategy as above and show that \eqref{small} holds for large $a$
and something similar when $a$ is small. Suppose then that $p=1/\om$ where $\om=\om(n)\to\infty$.
First consider those $a$ for which $ap\geq 1/\om^{1/2}$. In this case $p(a)\geq (1-e^{-2ap})/2$.
Thus,
\begin{multline*}
\sum_{\substack{ap\geq 1/\om^{1/2}\\a\leq np}}\binom{n}{a} \Pr\left(Bin(n-a,p(a))\le pn-a\right)\\
=\sum_{\substack{ap\geq 1/\om^{1/2}\\a\leq np}}e^{O(n\log\om/\om)}  e^{-\Omega(n/\om^{1/2})}=o(1).
\end{multline*}
If $ap\leq 1/\om^{1/2}$ then $p(a)= ap(1+O(ap))$. 
Then
\begin{multline}\label{a}
\sum_{\substack{ap<1/\om^{1/2}\\2\leq a\leq np}}\binom{n}{a} \Pr\left(Bin(n-a,p(a))\le pn-a\right)\\
\leq \sum_{\substack{ap<1/\om^{1/2}\\2\leq a\leq np}}\brac{\frac{ne}{a}e^{-np/10}}^a=o(1)
\end{multline}
provided $np\geq 11\log n$.

If $np\leq \log n-\log\log n$ then $G=G_{n,p}$ has isolated vertices asymptotically almost surely
and then $f(G)=1$. So we are left with the case where $\log n-\log\log n\leq np\leq 11\log n$.

We next observe that if there is a set $A$ for which
$2\leq |A|$ and $|A|+|B(A)|\leq np$ then there is a minimal size such set. Let $H_A=(A,E_A)$ be a graph
with vertex set $A$ and an edge $(v,w)\in E_A$ if and only if $v,w$ have a common neighbor in $G$. $H_A$ must be 
connected, else $A$ is not minimal. So we can find $t\leq a-1$ vertices $T$ such that $A\cup T$ spans 
at least $t+a-1$ edges between~$A$ and~$T$. Thus we can replace the estimate \eqref{a} by
\begin{align*}
\sum_{\substack{ap<1/\om^{1/2}\\2\leq a\leq np}}\sum_{t=1}^{a-1}&\binom{n}{a}\binom{n}{t}\binom{ta}{t+a-1}p^{t+a-1}
\Pr\left(Bin(n-a-t,p(a))\le pn-a\right)\\
&\leq\sum_{\substack{ap<1/\om^{1/2}\\2\leq a\leq np}}\sum_{t=1}^{a-1}\bfrac{ne}{a}^a\bfrac{ne}{t}^t\bfrac{taep}{t+a-1}^{t+a-1}e^{-anp/10}\\
&\leq \frac{1}{e^2np}\sum_{\substack{ap<1/\om^{1/2}\\2\leq a\leq np}}a\brac{(e^2np)^2e^{-np/10}}^a=o(1).
\end{align*}

\subsection{$p_0 < p<1$}
First let us define the constant $p_0$.  Let 
\begin{equation}\label{const:p}
p_0\approx 0.8941512242051071\ldots
\end{equation}
be a root of $2p-2p^2=\lambda_0$.
For the upper bound let $A=\{x,y\}$, where $x$ and $y$ satisfy $|N(x)\bigtriangleup N(y)| \le |N(x')\bigtriangleup N(y')|$ for any $x',y'\in 
V(G)$. Then $B=B(A) = N(x)\bigtriangleup N(y)$, and thus, asymptotically almost surely $|B|\leq (2p-2p^2)n$ plus a negligible error term~$o(n)$. (We may assume that $1-p=\Omega\left(\frac{\log n}{n}\right)$; for otherwise  we have two vertices of degree~$n-1$ with high probability, and hence, $f(G)$=2.)

To show the lower bound it is enough to prove that 
\begin{equation*}
\sum_{3\le a\le (2p-2p^2)n} \binom{n}{a}\Pr\left(Bin(n-a,p(a))\le (2p-2p^2)n-a\right)=o(1).
\end{equation*}
Indeed, this implies that if $|A|+|B| \le (2p-2p^2)n$, then $|A|=1$ or~$2$. But if $|A|=1$, then in a typical graph $|B| = (p+o(1))n >
(2p-2p^2)n$ since $p>1/2$. 

\subsubsection{$p$ Constant}
As in the previous section we split the sum into two sums for $3\le a\le \sqrt{n}$ and $\sqrt{n}<a\le pn$, respectively. Let 
$$\eps=1-\frac{(2p-2p^2)n-a}{(n-a)p(a)}\geq  1-\frac{2p-2p^2}{p(a)} >0.$$ 
To confirm the second inequality we have to consider two cases. The first one is for~$a$ odd and at least~$3$. Here,
$$1-\frac{2p-2p^2}{p(a)} > 1-\frac{2p-2p^2}{1/2} = (2p-1)^2 > 0.$$
The second case, for~$a$ even and at least~$4$, gives
$$1-\frac{2p-2p^2}{p(a)} > 1-\frac{2p-2p^2}{p(2)} = 0.$$
Now one can apply Chernoff bounds with the given~$\eps$ to show that 
$$
\sum_{3\le a < \sqrt{n}} \binom{n}{a}\Pr\left(Bin(n-a,p(a))\le (2p-2p^2)n-a\right) =o(1).
$$
Now we bound the second sum corresponding to $\sqrt{n}<a\le (2p-2p^2)n$. Note that  
\begin{align*}&\sum_{\sqrt{n}\le a \le (2p-2p^2)n} \binom{n}{a}\Pr\left(Bin(n-a,p(a))\le (2p-2p^2)n-a\right)\\
&=\sum_{\sqrt{n}\le a \le (2p-2p^2)n} \binom{n}{a}\Pr\left(Bin\left(n-a,\frac12+O(e^{-\Omega(n^{1/2})})\right)\le (2p-2p^2)n-a\right)\\
&\leq n2^{nh(2p-2p^2)+o(1)}=o(1)
\end{align*}
since $p>p_0$ implies that $2p-2p^2<\l_0$.

\subsubsection{$p=1-o(1)$}
One can check it by following the same strategy as above and in Section~\ref{sec:p_little_1}.

\subsection{$\lambda_0 \le p \le p_0$}
Let $\alpha =2\lambda_0/3$, $a=\lfloor \alpha n\rfloor$. Fix an $a$-set $A\subset V$ and generate
our random graph and determine $B=B(A)$ with $b=|B|$. Let $\e=(\log n)^4/\sqrt{n}$ and let $X_A$ be the indicator 
random variable for $a+b\le (\lambda_0+\e)n$ and $X=\sum_A X_A$. 
Then 
$$p(a)=\frac12+e^{-\Omega(n)}$$ 
and with $g_\l(\alpha)$ as defined in \eqref{defg},
\begin{equation}\label{EX}
E(X)=\exp\{(g_{\l_0+\eps}(2\l_0/3)+o(1))n\log 2\}.
\end{equation}
Now 
\begin{eqnarray*}
g_{\l+\eps}(\alpha)&=&g_{\l}(\alpha)+(1-\alpha)\brac{H\bfrac{\l+\eps-\a}{1-\a}-H\bfrac{\l-\a}{1-\a}}\\
&=&g_{\l}(\alpha)+\eps \log_2\bfrac{1-\l}{\l-\a}+O(\e^2).
\end{eqnarray*}
Plugging this into \eqref{EX} with $\l=\l_0$ and $\a=2\l_0/3$ we see that
\begin{equation}\label{EX2}
E(X)=\exp\set{\brac{\eps\log_2\bfrac{1-\l_0}{\l_0/3}+O(\e^2)}n\log 2}=e^{\Omega((\log n)^4n^{1/2})}.
\end{equation}

Next, we estimate the variance of $X$. We will argue that for $A,C\in \binom{V}{a}$ either
$|A\bigtriangleup C|$ is small (but the number of such pairs is small) or
$|A\bigtriangleup C|$ is large (but then the covariance $Cov(X_A,X_C)$ is very
small since if we fix the adjacency of some vertex $x$ to $C$, then the parity of
$|N(x)\cap (A\setminus C)|$ is almost a fair coin flip). Formally,  
$$
\begin{array}{rcl}
Var(X) &= E(X) &+ \quad\sum_{A\neq C} Cov(X_A, X_C) \\
         &\le E(X) &+ \quad\sum_{|A\bigtriangleup C|<2\sqrt{n}} \Pr(X_A=X_C=1)\\ 
         & &+    \quad\sum_{|A\bigtriangleup C|\ge 2\sqrt{n}, |A\cap C|\ge \sqrt{n}} Cov(X_A, X_C)\\
         & &+    \quad\sum_{|A\cap C|< \sqrt{n}} \Pr(X_A=X_C=1).
\end{array} 
$$
Since $E(X)$ goes to infinity, clearly $E(X)=o(E(X)^2)$. We show in Claims~\ref{clm:1}, \ref{clm:2} and~\ref{clm:3} that the remaining part is also bounded by $o(E(X)^2)$. 
Then 
Chebyshev's inequality will imply that $X>0$ asymptotically almost surely.

\begin{clm}\label{clm:1}
$\sum_{|A\bigtriangleup C|<2\sqrt{n}} \Pr(X_A=X_C=1)=o(E(X)^2)$
\end{clm}

\begin{proof}
We estimate trivially $\Pr(X_A=X_C=1)\le \Pr(X_A=1)$. Then,
\begin{align*}
\sum_{|A\bigtriangleup C|<2\sqrt{n}}\Pr(X_A=1) &= \binom{n}{a} \sum_{0\le i< \sqrt{n}} \binom{n-a}{i} \binom{a}{a-i} \Pr(X_A=1)\\
                 &= E(X) \sum_{0\le i< \sqrt{n}} \binom{n-a}{i} \binom{a}{a-i} \\
                 &\le E(X)\ 2^{O(\sqrt{n}\log n)}.
\end{align*}
Thus, \eqref{EX2} yields that $\sum_{|A\bigtriangleup C|<2\sqrt{n}} \Pr(X_A=X_C=1)=o(E(X)^2)$.
\end{proof}

\begin{clm}\label{clm:2}
$\sum_{|A\bigtriangleup C|\ge 2\sqrt{n}, |A\cap C|\ge \sqrt{n}} Cov(X_A, X_C) = o(E(X)^2)$
\end{clm}

\begin{proof}
If $x\in V\setminus (A\cup C)$, then $\Pr(x\in B(A)|X_C=1)=2^{-1+o(1/n)}$, 
since we can always find at least $\sqrt{n}$ vertices in~$A\setminus C$ 
with no adjacency with~$x$ determined by the event~$X_C=1$. Similarly, if $x\in C\setminus A$, then there are at 
least $\sqrt{n}-1$ vertices in $A\cap C$ such that their adjacency with $x$ is independent of the occurrence of $X_C=1$. This implies that 
$$\Pr(X_A=1|X_C=1)= \sum_{0\le i\le l-a} \binom{n-a}{i} 2^{-(n-a)+o(1)} = 2^{o(1)}\Pr(X_A=1),$$
and consequently, $Cov(X_A,X_C) = o\left(\Pr(X_A=1)^2\right)$. Hence,
\begin{align*}
\sum_{|A\bigtriangleup C|\ge 2\sqrt{n}, |A\cap C|\ge \sqrt{n}} Cov(X_A, X_C) \le \binom{n}{a}^2 o\left(\Pr(X_A=1)^2\right) =  o(E(X)^2).
\end{align*}
\end{proof}

\begin{clm}\label{clm:3}
$\sum_{|A\cap C|< \sqrt{n}} \Pr(X_A=X_C=1) = o(E(X)^2)$
\end{clm}

\begin{proof}
First let us estimate the number of ordered pairs $(A,C)$ for which $|A\cap C|<\sqrt{n}$. Note,
\begin{align}\label{eq:clm3:1}
\sum_{|A\cap C|<\sqrt{n}} {1}  &= \binom{n}{a} \sum_{0\le i<\sqrt{n}} \binom{n-a}{a-i}\binom{a}{i}\notag\\
                                              &\le \sqrt{n}\binom{n}{a}\binom{n-a}{a}\binom{a}{\sqrt{n}}\notag\\
                                              &= 2^{n\left(H(\alpha) + H\left( \frac{\alpha}{1-\alpha} \right)(1-\alpha) +o(1)\right)}.
\end{align}
Now we will bound $\Pr(X_A=X_C=1)$ for fixed $a$-sets $A$ and~$C$. Let $S\subset A\setminus C$ be a set of size $s=|S|=\lfloor \sqrt{n} \rfloor$. 
Define a new indicator random variable $Y$ which takes the value $1$ if and only if $|B(C)\setminus S|\le (\lambda_0+\e)n-a$. Clearly, 
$X_C\leq Y$ and
\begin{align*}
\Pr(Y=1) &= \Pr\left(Bin(n-a-s, p(a))\le (\lambda_0+\e)n-a\right)\\
       &\le 2^{s+o(1)} \sum_{0\le i\le (\lambda_0+\e)n-a} \binom{n-a}{i}2^{-(n-a)}\\
       &= 2^{s+o(1)} \Pr(X_A=1),
\end{align*}
Now if we condition on the existence or otherwise of all edges $F'$ between $C$ and $V\setminus S$ 
then if $x\in V \setminus A$
$$\Pr(x\in B(A)\mid F' \mbox{ and } F'') \in \left[\frac{1-(1-2p)^{s}}{2}, \frac{1+(1-2p)^{s}}{2}\right],$$
where $F''$ is the set of edges between~$x$ and $A\setminus S$. 
This implies that
\begin{align*}
\Pr(X_A=1|Y=1) &= \sum_{0\le i\le (\lambda_0+\e)n-a} \binom{n-a}{i}2^{-(n-a)+o(1)}\\
               &= 2^{o(1)} \Pr(X_A=1),
\end{align*}
Consequently, 
$$\Pr(X_A=X_C=1)\le \Pr(X_A=Y=1)\le 2^{\sqrt{n}+o(1)}\Pr(X_A=1)^2.$$ 
Hence, \eqref{eq:clm3:1} implies
$$
\sum_{|A\cap C|< \sqrt{n}} \Pr(X_A=X_C=1) \le 2^{n\left(H(\alpha) + H\left( \frac{\alpha}{1-\alpha} \right)(1-\alpha) +o(1)\right)} 
\Pr(X_A=1)^2.
$$
To complete the proof it is enough to note that 
$$E(X)^2=2^{n\left(2H(\alpha)+o(1)\right)} \Pr(X_A=1)^2$$ 
and 
$$2H(\alpha) > H(\alpha) + H\left( \frac{\alpha}{1-\alpha}\right)(1-\alpha).$$ 
Indeed, the last inequality follows from the strict concavity of the entropy 
function, since then
$(1-\alpha)H\left( \frac{\alpha}{1-\alpha}\right) + \alpha H(0) \le H(\alpha)$ with the equality for $\alpha=0$ only.
\end{proof}

Now we show that $f(G_{n,p})\ge (\lambda_0-\e)n$. 
We show that 
$$
\sum_{1\le a\le (\lambda_0-\e)n} \binom{n}{a} \Pr\left(Bin(n-a,p(a))\le (\lambda_0-\e)n-a\right)=o(1).
$$
As in previous sections we split this sum into two sums but this time we make the break into $1\le a\le (\log n)^2$ and $(\log n)^2<a\le 
(\lambda_0-\e)n$, respectively.
In order to estimate the first sum we use the Chernoff bounds with deviation $1-\theta$ from the mean where
$$\theta=1-\frac{(\lambda_0-\e)n-a}{(n-a)p(a)}\ge 1-\frac{\lambda_0-\e}{p(a)} \ge 1-\frac{\lambda_0-\e}{\lambda_0} =\frac{\e}{\l_0}.$$

Consequently,
\begin{align*}
\sum_{2\le a < (\log n)^2} \binom{n}{a}&\Pr\left(Bin(n-a,p(a))\le (\lambda_0-\e)n-a\right)\\ 
&\le (\log n)^2 \binom{n}{(\log n)^2} \exp\{ -\Omega(\log n)^4\}\\ 
&\le \exp \{ -\Omega(-(\log n)^4)\}=o(1).
\end{align*}
Now we bound the second sum corresponding to $\log\log n<a\le (\lambda_0-\e)n$. 
\begin{multline*}
\sum_{\log\log n\le a \le (\lambda_0-\e)n} \binom{n}{a}\Pr\left(Bin(n-a,p(a))\le (\lambda_0-\e)n-a\right)\\
=2^{n(h(\l_0-\eps)+o(1/n))}=o(1).
\end{multline*}

\section{General Graphs}\label{GG}
Here we present the proof of Theorem~\ref{thm:2}. First, we prove a weaker result $f(n)\le (0.440\ldots+o(1)) n$.

Suppose we aim at showing that $f(n)\le \lambda n$. We fix some $\alpha$ and
$\rho$ and let $a=\alpha n$ and $r=\rho n$. For each $a$-set $A$ let $R(A)$ 
consist of all sets that have Hamming distance at most~$r$ from~$B(A)$. If 
\begin{equation}\label{eq:vol}
\binom{n}{a}\sum_{i=0}^{r}\binom{n}{i}=2^{n(H(\a)+H(\r)+o(1))}>2^n,
\end{equation}
then there are $A,A'$ such that
$R(A)\cap R(A')\ni C$ is non-empty. This means that $C$ is within Hamming distance $r$ from
both $B=B(A)$ and $B'=B(A')$. Thus $|B\bigtriangleup B'|\le 2r$. 

Let all vertices in $A''=A\bigtriangleup A'$ flip their neighbors, \textit{i.e.}, execute operation $O_1$. 
The only vertices outside of~$A''$ that
can have an odd number of neighbors in $A''$ are restricted to
$(B\bigtriangleup B')\cup (A\cap A')$. Thus
\begin{equation}\label{eq:h} 
f(G)\le |A\bigtriangleup A'|+|(B\bigtriangleup B')\cup (A\cap A')|\le 2a+2r = 2n(\alpha+\rho).
\end{equation}
Consequently, we try to minimise $\alpha+\rho$ subject to $H(\alpha)+H(\rho)> 1$. Since
the entropy function is strictly concave, the optimum satisfies $\alpha=\rho$, otherwise
replacing each of $\alpha,\rho$ by $(\alpha+\rho)/2$ we strictly increase
$H(\alpha)+H(\rho)$ 
without changing the sum. 
Hence, the optimum choice is 
\begin{equation*}
\alpha=\rho\approx 0.11002786443835959\ldots
\end{equation*}
the smaller root of
$H(x)=1/2$, proving that $f(n)\le (0.440\ldots+o(1)) n$. 

In order to obtain a better constant we modify the approach taken in~\eqref{eq:vol}.
Let us take $\delta=0.275$,
$\alpha=0.0535$, $a=\lfloor\alpha n\rfloor$, $d=\lfloor \delta n\rfloor$. Look
at the collection of sets $B(A)$, $A\in {[n]\choose a}$. This gives ${n\choose
 a}= 2^{n(H(\alpha)+o(1))}$ binary $n$-vectors.

We claim that some two of these vectors are at distance at most $d$. If not,
then inequality $(5.4.1)$ in~\cite{Lint} says that
$$
H(\alpha)+o(1)\le \min\{1+g(u^2)-g(u^2+2\delta u + 2\delta) : 0\le u\le
1-2\delta\},
$$
where $g(x)=H((1-\sqrt{1-x})/2)$.  
In particular, if we take $u=1-2\delta=0.45$, we get $0.30108+o(1)\le
0.30103$, a contradiction. 

Thus, we can find two different $a$-sets $A$ and $A'$ such that
$|B(A)\bigtriangleup B(A')|\le d$. As in~\eqref{eq:h}, we can conclude
that $f(G)\le 2a+d\le (0.382+o(1))n$.

\section{Acknowledgment}
The authors would like to thank Shiang Yong Looi for suggesting this problem.

\end{document}